\documentclass[letterpaper, 10 pt, conference]{ieeeconf}  

\IEEEoverridecommandlockouts                              
\usepackage{amsmath}
\usepackage{amssymb}
\usepackage{bm}
\usepackage{nicefrac}

\usepackage{graphicx}
\usepackage{booktabs}

\usepackage{hyperref}
\usepackage{cleveref}
\usepackage{acronym}

\usepackage{graphicx}
\usepackage{tikz}
\usepackage{pgfplots}
\pgfplotsset{compat=1.18}
\usepackage{xcolor}
\usepgfplotslibrary{groupplots}

\usepackage[utf8]{inputenc}

\usepackage{subcaption}

\usepackage{arydshln}

\usepackage[normalem]{ulem}
\newcommand{\rev}[1]{\textcolor{black}{#1}}
\newcommand{\del}[1]{}

\usepackage{ifthen}
\newboolean{arxiv}
\setboolean{arxiv}{true}

\newtheorem{theorem}{Theorem}
\newtheorem{lemma}[theorem]{Lemma}
\newtheorem{proposition}[theorem]{Proposition}

\newtheorem{assumption}{Assumption}

\newtheorem{example}{Example}

\acrodef{OCO}{Online Convex Optimization}
\acrodef{SDP}{Semi-Definite Program}
\acrodef{IQC}{Integral Quadratic Constraint}
\acrodef{ADMM}{Alternating Direction Method of Multipliers}

\title{\LARGE \bf
Online Convex Optimization and Integral Quadratic Constraints: \\ \rev{An automated} approach to regret analysis
}

\author{Fabian Jakob and Andrea Iannelli
\thanks{
Authors are with the University of Stuttgart, Institute for Systems Theory and Automatic Control, 70569 Stuttgart, Germany.
{\tt\small \{fabian.jakob, andrea.iannelli\}@ist.uni-stuttgart.de}.
\rev{Fabian Jakob acknowledges the support of the International Max Planck Research School for Intelligent Systems (IMPRS-IS)}.
} 
}

\begin{document}

\maketitle
\thispagestyle{empty}
\pagestyle{empty}

\begin{abstract}

We propose a novel approach for analyzing dynamic regret of first-order \del{constrained} online convex optimization \rev{(OCO)} algorithms for strongly convex and Lipschitz-smooth objectives. Crucially, we provide a general analysis that is applicable to a wide range of first-order algorithms that can be expressed as an interconnection of a linear dynamical system in feedback with a first-order oracle. By leveraging Integral Quadratic Constraints (IQCs), we derive a semi-definite program which, when feasible, provides a regret guarantee for the online algorithm. For this, the concept of variational IQCs is introduced as the generalization of IQCs to time-varying monotone operators. Our bounds capture the temporal rate of change of the problem in the form of the path length of the time-varying minimizer and the objective function variation. In contrast to standard results in OCO, our results do not require \rev{neither} the assumption of gradient boundedness, nor that of a bounded feasible set. Numerical analyses showcase the ability of the approach to capture the dependence of the regret on the function class condition number. 
\end{abstract}

\section{INTRODUCTION}

\ac{OCO} has emerged as a powerful framework for tackling real-time decision-making problems under uncertainty. Traditionally, the study of OCO has focused on proposing online algorithms whose performance is assessed in terms of their static or dynamic regret \cite{zinkevich,jadbabaie}. 
In recent years, this framework has raised interest in the control community both for the design of OCO-inspired controllers \cite{nonhoff} and for using the concept of regret as a performance metric to evaluate controllers dealing with uncertainty \cite{karapetyan}.

OCO algorithms aim to make a sequence of decisions in real-time minimizing a cumulative loss function that is revealed sequentially. Several algorithms have emerged, among which first-order algorithms represent a fundamental subclass. Each algorithm comes with its individual regret guarantees and proof techniques to verify them \cite{hazan_paper}. For the particular case of strongly convex and smooth objective functions, also accelerated methods \cite{hu_neurips} and multi-step methods with regret guarantees have been proposed \cite{zhao_improved_analysis}. 
However, a general methodology to approach their analysis does not exist. 

In contrast, for static convex optimization, \rev{an automated} approach to analyze first-order algorithms based on systems theory has been thoroughly investigated in the last years, see e.g. \cite{lessard,michalowski,scherer_ebenbauer} and references therein. The idea is to model the first-order algorithm as a Lur'e system, i.e. an interconnection of a linear \rev{dynamical} system in feedback with a \rev{first-order oracle}. Integral Quadratic Constraints (IQCs) \rev{can then be} leveraged to formulate analysis conditions based on Semi-Definite Programs (SDPs). Recently, this framework has also been extended to time-varying optimization \del{, providing an approach to establish tracking guarantees of a time-varying minimizer}\cite{jakob_iannelli}.

Time-varying optimization and OCO are in fact inherently related \cite{dallanese_information_streams}. In this paper, we take up tools from \cite{jakob_iannelli} to propose \rev{an automated} approach to dynamic regret analysis for first-order algorithms in OCO.
By leveraging a system theoretic view on algorithms, we model an OCO algorithm as a dynamical system interconnected with a \emph{time-varying} first-order oracle. We consider strongly-convex and smooth objective functions. To handle the time-varying nature of the oracle, we leverage variational Integral Quadratic Constraints (vIQCs), which in contrast to conventional IQCs \rev{depend on different measures of temporal variation of the problem.} \del{not constrain the quadratic term to be greater than zero, but than a term capturing the temporal variation} In line with \rev{the OCO literature}\del{standard results in OCO \cite{tinbao_yang,zhao_improved_analysis},} we obtain a bound on the dynamic regret that accounts for the path length of the time-varying optimal solution, the objective function variation, \rev{and the gradient variation}. In contrast to the common analysis approaches in OCO, our proofs are \rev{algorithm-agnostic and thus, applicable to a large number of first-order algorithms}. The regret upper bound depends on decision variables of the SDP, leaving the possibility to tune the bound by trading off their magnitude. We show the implicit dependence of the regret on the function class condition \rev{ratio} via a numeric case study.

The main contributions can be summarized as follows. We provide a general modeling framework for OCO algorithms \del{-capable of handling both single-step and multi-step algorithms-}and a new proof technique for bounding dynamic regret. Notably, our approach does neither require the typical bounded gradient assumption, nor does it require the bounded feasible set assumption with the use of vIQCs. A numerical study demonstrates the versatility of the general algorithm formulation and provides a comparative study of commonly used OCO-algorithms. We believe the strengths of this new approach consist of: weaker assumptions required for the analysis; generality; and insights provided by comparing the analysis results of different algorithms.

The paper is structured as follows. We state the preliminaries in section~\ref{sec:problem_statement}, introducing the notion of regret and our algorithm formulation. In section~\ref{sec:IQC_intro}, we introduce the IQC formulation needed to establish our regret bounds and illustrate them with numerical examples. In section~\ref{sec:main}, we derive our regret bounds as main results. 
We conclude the paper in section~\ref{sec:conclusions}.

\textbf{Notation.} Let $\mathrm{vec}(v_1, v_2)$ denote the vertical stack of the vectors $v_1, v_2$. We denote $1_p$ a column vector of ones of length $p$ and $\mathbb{I}_p$ the index set of integers from $1$ to $p$. We indicate by $I_d$ a $d\times d$ identity matrix. For positive definite $P$, we define the norm $\| v \|_P := \sqrt{v^\top P v}$. The Kronecker product between two matrices is written as $A \otimes B$. Let $\mathrm{diag}(v)$ denote the diagonal matrix with the elements of a vector $v$ and $\mathrm{blkdiag}(P_1,P_2)$ a blockdiagonal matrix of $P_1$, $P_2$. We write $f(T) = \mathcal{O}(g(T))$ if $\lim_{T \rightarrow \infty} \frac{f(T)}{g(T)} < \infty$. A linear dynamic mapping $x_{t+1} = A x_t + B u_t,\, y_t = C x_t + D u_t$ is compactly expressed as $y_t = G u_t,$ $G=$ \footnotesize$\left[
    \begin{array}{@{\hskip 0.1em}c@{\hskip 0.1em}|@{\hskip 0.1em}c@{\hskip 0.1em}}
    A & B \\ \hline
    C & D
    \end{array}
\right]$. \normalsize
\section{Problem Setup}
\label{sec:problem_statement}

\subsection{Online Convex Optimization}

In \rev{OCO}, an algorithm sequentially selects an action $x_t \in \mathcal{X}$ from a closed convex decision set $\mathcal{X}$ at each time step $t$, based solely on the information available up to time $t-1$. Upon choosing $x_t$, a convex objective function $f_t$ is revealed, and the algorithm incurs a loss $f_t(x_t)$. Throughout this work we assume the objective functions $f_t$ are $m$-strongly convex and $L$-smooth, and that $\mathcal{X} \subseteq \mathbb{R}^d$. The goal is to minimize the cumulative loss over $T$ rounds, and the performance of an OCO algorithm is typically assessed in terms of regret, which represents the cumulative suboptimality \emph{w.r.t.} the best possible decisions in hindsight \cite{hazan_book}.
In this work, we define the best hindsight decision as the pointwise-in-time minimizer
\begin{equation}
    \label{eq:oco}
    x_t^\star = \arg \min_{x \in \mathcal{X}} f_t(x), \quad t\in \mathbb{N}_+,
\end{equation}
leading to the notion of dynamic regret \cite{jadbabaie}
\begin{equation}
	\label{eq:regret_def}
    \mathcal{R}_T = \sum_{t=1}^{T} f_t(x_t) - \sum_{t=1}^{T} f_t(x_t^\star).
\end{equation}

Classical analyses of OCO primarily focus on deriving upper bounds on the regret. Typically, dynamic regret bounds are characterized in terms of regularity measures, which quantify the temporal change of the problem \cite{tinbao_yang, mokhtari}. Notable regularity measures commonly employed in the literature include the path length and squared path length, respectively defined as
\begin{equation}
    \mathcal{P}_T = \sum_{t=1}^{T} \| x_{t}^\star - x_{t+1}^\star \|, \quad \mathcal{S}_T = \sum_{t=1}^{T} \| x_{t}^\star - x_{t+1}^\star \|^2,
\end{equation}
and the function \rev{and gradient variation}
\begin{align}
    \label{eq:function_variation}
    \mathcal{V}_T &= \sum_{t=2}^{T} \sup_{x \in \mathcal{X}} | f_{t-1}(x) - f_t(x) |, \\
	\label{eq:gradient_variation}
	\rev{
	\mathcal{G}_T} &= \rev{\sum_{t=2}^{T} \sup_{x \in \mathcal{X}} \| \nabla f_{t-1}(x) - \nabla f_t(x) \|^2}.
\end{align}

Table~\ref{tab:literature} provides a comparative overview of existing regret bounds for strongly convex and smooth objectives, alongside the bounds established in this work. 
\begin{table}
	\centering
	\setlength{\tabcolsep}{4pt}
	\caption{Regret bounds for strongly convex and smooth objectives {\rev{\footnotesize(OGD=online gradient descent, MD=mirror descent)}}.}
	\label{tab:literature}
	\begin{tabular}{llll}
		\toprule
		Ref. & Algorithm & Regret bound & Assumptions \\
		\midrule
		\cite{mokhtari} & OGD & $\mathcal{O}(\mathcal{P}_T)$ & $\nabla f_t$ bounded \\
		\rev{\cite{zhao_gradient_variation}} & \rev{Optimistic MD} & \rev{$\mathcal{O}(\log \mathcal{G}_T)$} & \rev{$\mathcal{X}$ bounded} \\ 
		\cite{zhao_improved_analysis} & Multi-step OGD & $\mathcal{O}\left( \min\{ \mathcal{P}_T, \mathcal{S}_T, \mathcal{V}_T \} \right)$ & $\nabla f_t$ bounded \\
		\midrule
		Thm. \ref{thm:static} & (\ref{eq:generalized_algorithm}) & $\mathcal{O}(\mathcal{P}_T)$ & $\mathcal{X}$ bounded \\
		Thm. \ref{thm:dynamic} & (\ref{eq:generalized_algorithm}) & $\rev{\mathcal{O}(\mathcal{S}_T + \mathcal{V}_T + \mathcal{G}_T)}$ & None \\
		\bottomrule
	\end{tabular}
\end{table}

We will consider (\ref{eq:oco}) in its equivalent composite form
\label{eq:composite_optimization}
\begin{equation}
    \min_{x \in \mathbb{R}^d} f_t(x) + \mathcal{I}_\mathcal{X}(x),
\end{equation}
where $\mathcal{I}_\mathcal{X}$ is the indicator function
\begin{equation*}
	\mathcal{I}_\mathcal{X}(x) = \begin{cases}
		0 &, \text{if } x\in \mathcal{X}, \\
		\infty &, \text{if } x\notin \mathcal{X}.
	\end{cases}
\end{equation*}
As it is well known, the subdifferential of $\mathcal{I}_\mathcal{X}$ is the normal cone of the set $\mathcal{X}$, denoted as $\partial I_\mathcal{X}(x) \triangleq \mathcal{N}_\mathcal{X}(x)$ \cite{beck}.

\subsection{General First-Order Algorithms}
\label{sec:generalized_algorithms}

In line with \cite{lessard}--\cite{scherer_ebenbauer}, we consider general first-order algorithms that are expressed as a linear time-invariant system
\begin{subequations}
\label{eq:generalized_algorithm}
\begin{align}
	\label{eq:generalized_algorithm:state_space}
	\begin{aligned}
		\xi_{t+1} &= A \xi_t + B u_t, \\
		y_t &= C \xi_t + D u_t
	\end{aligned}
\end{align}
with state $\xi \in \mathbb{R}^{n_\xi}$ in feedback with a first-order oracle $u_t = \varphi_t(y_t)$.
For some integer $p \geq 1, q\geq 0$, we consider that the algorithm makes use of $p$ gradient evaluations of $f_t$ and $q$ subgradient evaluations of $\mathcal{I}_\mathcal{X}$, and the input and output can be decomposed into
\begin{equation}
	y_t = \begin{bmatrix} s_t \\ z_t \end{bmatrix}, \qquad u_t = \begin{bmatrix}
		\delta_t \\ g_t
	\end{bmatrix},
\end{equation}
where $s_t := \mathrm{vec}(s_t^1, \dots, s_t^{p})$ and $z_t := \mathrm{vec}(z_t^1, \dots, z_t^{q})$, with $s_t^i, z_t^j \in \mathbb{R}^d$ for all $i \in \mathbb{I}_p, j\in \mathbb{I}_q$, and
\begin{equation}
	\label{eq:generalized_algorithm:gradient}
	\delta_t=\left[\begin{smallmatrix} \nabla f_t(s^1_t) \\ \vdots \\ \nabla f_t ( s^{p}_t ) \end{smallmatrix}\right], \quad g_t=\left[\begin{smallmatrix} g^{1}_t \\ \vdots \\ g^{q}_t \end{smallmatrix}\right], \quad g_t^j \in \mathcal{N}_\mathcal{X}(z_t^j).
\end{equation}
\end{subequations}
The case $q=0$ is relevant for unconstrained optimization problems.
The algorithm's iterate coincides with the first output, i.e., $x_t =: s_t^1$. To enforce that $x_t$ only depends on information up to $t-1$, we assume the readout is independent from $u_t$, that is, the first block-row in $D$ is assumed to be zero.
The first block-row of $C$ is denoted as $C_{1}$, i.e., $x_t = C_1 \xi_t$. We furthermore \rev{assume that the pair $(C_1, A)$ is observable, so that if the set $\mathcal{X}$ is bounded, then the set of internal states $\xi_t$ is also bounded.}


\rev{To ensure that this dynamical system is meaningful from an optimization standpoint} we require that the fixed point of (\ref{eq:generalized_algorithm}) satisfies the first-order optimality conditions of (\ref{eq:composite_optimization})
\begin{equation}
	\label{eq:first_order_optimality}
	- \nabla f_t (x_t^\star) \in \mathcal{N}_\mathcal{X}(x_t^\star).
\end{equation}
\rev{We will particularly constrain ourselves to algorithms whose fixed-points are of the form}
\begin{align} 
\begin{aligned}
	\label{eq:fixed_point}
	\xi_t^\star &= A \xi_t^\star, & y_t^\star &= C \xi_t^\star, \\
	y_t^\star &= \begin{bmatrix} 1_p \otimes I_d \\ 1_q \otimes I_d \end{bmatrix} x_t^\star,
	& u_t^\star &= \begin{bmatrix} 1_p \otimes I_d \\ -1_q \otimes I_d \end{bmatrix} \nabla f_t(x_t^\star),
\end{aligned}
\end{align}
that is, $s^{i,\star}_t = z^{j,\star}_t = x_t^\star$, and $\delta^{i,\star}_t = - g_t^{j,\star} = \nabla f_t(x_t^\star)$, for all $i\in\mathbb{I}_p\, j\in\mathbb{I}_q$. \rev{Moreover, we enforce
$B u_t^\star = 0$ and $D u_t^\star = 0$, which will be necessary for subsequent derivations.} We make the following assumptions to ensure (\ref{eq:generalized_algorithm}) exhibits (\ref{eq:fixed_point}) as fixed-point.
\begin{assumption}
\label{assum:fixed_point}
	There exists a matrix $U$ such that
	\begin{equation}
	\label{eq:assum:fixed_point:michalowski}
	\begin{bmatrix}
		I - A \\ C
	\end{bmatrix} U = \begin{bmatrix}
		0 \\ \bigl[ \begin{smallmatrix} 1_p \\ 1_q \end{smallmatrix} \bigr] \otimes I_d
	\end{bmatrix},
	\end{equation}
\end{assumption}
\begin{assumption}
	\label{assum:Bu_Du}
    If $q \geq 1$, then the kernels of the matrices $B$ and $D$ satisfy
	\begin{equation}
		\label{eq:assum:fixed_point:B_D}
		\mathrm{ker}\, B =  \begin{bmatrix} 1_p \otimes I_d \\ -1_q \otimes I_d \end{bmatrix}, 
		\quad
		\mathrm{ker}\, D =  \begin{bmatrix} 1_p \otimes I_d \\ -1_q \otimes I_d \end{bmatrix}.
	\end{equation}
\end{assumption}
Both assumptions essentially allow us, given an optimal point $x_t^\star$, to reconstruct the optimal algorithmic state as $\xi^\star_t = U x_t^\star$, which represents a special case of the conditions worked out in \cite{upadhyaya}.  Assumption~\ref{assum:Bu_Du} specifically ensures $B u_t^\star = 0$, $D u_t^\star=0$ \cite{jakob_iannelli}.
This is slightly restrictive, however, we still cover a large class of classical \ac{OCO} algorithms. 
\rev{We illustrate the algorithm representation with the following example.}
\rev{
\begin{example}
	\label{example:multi_step}
	Consider the example of the following two-step projected gradient descent
	\begin{align}
		\label{eq:example:two_step}
		\begin{aligned}
		\hat{x}_t &= \Pi_\mathcal{X} \left[x_t - \alpha \nabla f_t(x_t) \right]\\
		x_{t+1} &= \Pi_\mathcal{X} \left[ \hat{x}_t - \alpha \nabla f_t(\hat{x}_t) \right],
		\end{aligned}
	\end{align}
	with $\Pi_\mathcal{X}(z) := \arg\min_{x\in\mathcal{X}} \| x - z \|^2$ and $\alpha > 0$.
	To bring (\ref{eq:example:two_step}) into the form (\ref{eq:generalized_algorithm}) we leverage the first-order optimality condition arising from the projections, namely
	\begin{align}
		\label{eq:example:first_order_opt}
		\begin{aligned}
		&x_t - \alpha \nabla f_t(x_t) - \hat{x}_t \in \alpha \, \mathcal{N}_\mathcal{X}(\hat{x}_{t}) \\
		&\hat{x}_t - \alpha \nabla f_t(\hat{x}_t) - x_{t+1} \in \alpha \, \mathcal{N}_\mathcal{X}(x_{t+1}).
		\end{aligned}
	\end{align}
	As a technicality, we scaled the cones by $\alpha$ so that (\ref{eq:assum:fixed_point:B_D}) will be satisfied.
	Now define the outputs $s_t\triangleq\mathrm{vec}(x_t, \hat{x}_t)$, $z_t \triangleq \mathrm{vec}(\hat{x}_t, x_{t+1})$ and the inputs $\delta_t\triangleq\mathrm{vec}\left(\nabla f_t(x_t),\nabla f_t(\hat{x}_t)\right)$, $g^1_t \in \mathcal{N}_{\mathcal{X}}(\hat{x}_t)$, $g^2_t \in \mathcal{N}_{\mathcal{X}}(x_{t+1})$. Then, by letting $\xi_t := x_t$, we can write (\ref{eq:example:two_step}) as (\ref{eq:generalized_algorithm}) with $p=2, q=2$ and 
	\begin{align*}
		A&= 1 \otimes I_d, & B&=\begin{bmatrix} -\alpha & -\alpha & -\alpha & -\alpha \end{bmatrix}\otimes I_d, \\
		C&=\begin{bmatrix} 1 \\ 1 \\ 1 \\ 1 \end{bmatrix} \otimes I_d, & D&= \begin{bmatrix} 0 & 0 & 0 & 0 \\ -\alpha & 0 & -\alpha & 0 \\ -\alpha & 0 & -\alpha & 0 \\
		-\alpha & -\alpha & -\alpha & -\alpha \end{bmatrix} \otimes I_d.
	\end{align*}
\end{example}
}

\rev{In the spirit of Example~\ref{example:multi_step}, many relevant algorithms} can be formulated as (\ref{eq:generalized_algorithm}), such as Online Gradient Descent (O-GD) \cite{zinkevich}, Online Accelerated Gradient Descent (O-AGD) \cite{tinbao_yang}, the Online Nesterov Method (O-NM) \cite{hazan_book} or Online Mirror Descent \cite{hall}.




\section{IQCs for varying operators}
\label{sec:IQC_intro}

Integral Quadratic Constraints provide a framework to characterize unknown or nonlinear input-output mappings in terms of inequalities \cite{megretzki}. Informally, an operator $\varphi$ satisfies an IQC defined by a dynamic filter $\Psi$ and symmetric matrix $M$, if for all square summable sequences $y$ and $u = \varphi(y)$ it holds 
\begin{equation}
    \label{eq:IQC_classic}
    \sum_{t=1}^T \psi_t^\top M \psi_t \geq 0, \quad \psi_t = \Psi \begin{bmatrix} y_t \\ u_t \end{bmatrix}
\end{equation}
for all $T \geq 1$.
Such descriptions have proven particularly useful in the context of first-order algorithms for static optimization \cite{lessard,scherer_ebenbauer}. We will introduce an adapted formulation to characterize the input-output relation of time-varying gradients, which will be leveraged to establish our regret bounds.

\subsection{Pointwise IQCs}
\label{sec:pointwise_IQC}

Many IQCs have been derived for gradients of convex and strongly-convex-smooth functions \cite{lessard}. For time-varying operators, we can recover their pointwise-in-time formulation.

\begin{lemma}
\label{lemma:sector_IQC}
Let $f_t$ be $m$-strongly convex and $L$-smooth. Take $x_t \in \mathcal{X}$, $x_t^\star$ as in (\ref{eq:oco}) and define 
\begin{subequations}
    \label{eq:sector_quasi_IQC}
\begin{equation}
    \psi_t = \begin{bmatrix} L I_d & -I_d \\ -mI_d & I_d \end{bmatrix} \begin{bmatrix} x_t-x_t^\star \\ \nabla f_t(x_t) \end{bmatrix}
\end{equation}
Then for all $t$, it holds
\begin{equation}
    \label{eq:sector_quasi_IQC:ineq}
	\psi_t^\top M_1 \psi_t \geq 0 \quad \text{with} \quad M_1 = \begin{bmatrix}
		0 & 1 \\
		1 & 0 \\
	\end{bmatrix} \otimes I_d.
\end{equation}
\end{subequations}
\end{lemma}

Lemma~\ref{lemma:sector_IQC} is a simple consequence of \cite[Prop. 5]{lessard}.
Notably, (\ref{eq:sector_quasi_IQC}) satisfies the inequality \emph{pointwise}, i.e. the filter $\Psi$ is static and every single summand is nonnegative.

We will also need a similar characterization of $\mathcal{I}_\mathcal{X}$.
\begin{lemma}
    \label{lemma:sector_IQC_h}
    Take $x_t \in \mathcal{X}$ and $x_t^\star$ as in (\ref{eq:oco}). Define $\hat{\psi}_t = \mathrm{vec} ( x_t-x_t^\star, \beta_t)$
    for some $\beta_t \in \mathcal{N}_\mathcal{X}(x_t)$. Then for $M_1$ as in (\ref{eq:sector_quasi_IQC:ineq}) and all $t$, it holds $\hat{\psi}_t^\top M_1 \hat{\psi}_t \geq 0$.
\end{lemma}

\rev{Note that} Lemma~\ref{lemma:sector_IQC_h} is simply the statement that the normal cone of a convex set is a monotone operator.

\subsection{Variational IQCs}
\label{sec:vIQC_def}

It is well known that introducing a dynamic mapping from $(x_t-x_t^\star, \nabla f_t(x_t))\mapsto \psi_t$ can lead to a less conservative input-output characterization of the gradient operator \cite{lessard,scherer_ebenbauer}. Unfortunately, most dynamic IQC results are not applicable to time-varying operators. However, recently the notion of variational IQCs has been introduced \cite{jakob_iannelli}.

\begin{proposition}
    \label{prop:vIQC}
    \begin{subequations}
    \label{eq:IQC:delta}
    Let $f_t$ be $m$-strongly convex and $L$-smooth. \rev{Define the variational measures $\Delta x_t^\star = x_t^\star - x_{t+1}^\star$ and $\Delta \delta_t (\cdot) = \nabla f_t(\cdot) - \nabla f_{t+1}(\cdot)$. Consider the mapping
    \begin{align}
        \psi_t = \Psi_f
		\begin{bmatrix}
            x_t - x_t^\star \\ \nabla f_t(x_t) \\ \Delta x_t^\star \\ \Delta \delta_t(x_t)
		\end{bmatrix}
    \end{align}
    with the linear filter
        \begin{equation}
            \label{eq:vIQC_realization}
            \Psi_f =
            \left[ 
            \begin{array}{cccc|cccc} 
                0 & 0 & 0 & 0 & I_d & 0 & I_d & 0 \\ 
                0 & 0 & 0 & 0 & 0 & I_d & 0 & -I_d \\ 
                0 & 0 & 0 & 0 & -mI_d & I_d & 0 & 0 \\ 
                0 & 0 & 0 & 0 & aI_d & 0 & 0 & 0 \\ 
                \hline 
                - LI_d & I_d & 0 & 0 & LI_d & -I_d & 0 & 0\\ 
                0 & 0 & 0 & 0 & -mI_d & I_d & 0 & 0 \\ 
                0 & 0 &  I_d & 0 & 0 & 0 & 0 & 0 \\ 
                a  I_d & 0 & 0 & 0 & 0 & 0 & 0 & 0 \\ 
                0 & 0 & 0 &  I_d & 0 & 0 & 0 & 0 \\ 
                -m I_d &  I_d & 0 & 0 & 0 & 0 & 0 & 0 
            \end{array} 
            \right],
        \end{equation}
    where $a := \sqrt{\frac{m(L-m)}{2}}$.}
    Then the quadratic inequality
    \rev{
    \begin{equation}
        \label{eq:IQC:delta:quadratic_inequality}
        \sum_{t=1}^{T} 
        \psi_t^\top 
        M_2 
        \psi_t 
        \geq 
        - 4(L-m) \, \mathcal{V}_T
    \end{equation}
    }
    holds with $\mathcal{V}_T$ as defined in (\ref{eq:function_variation}) and the blockdiagonal matrix $M_2 = \mathrm{blkdiag}\left(\frac{1}{2}M_1, \bigl[\begin{smallmatrix}I_d &\\&-I_d \end{smallmatrix}\bigr], \frac{1}{2}\bigl[\begin{smallmatrix}I_d &\\&-I_d \end{smallmatrix}\bigr]\right)$.
    \end{subequations}
\end{proposition}
\begin{proof}
    \rev{The proof is the result of applying \cite[Prop. 4.2]{jakob_iannelli} with $\rho=1$ and leveraging the definition of $\mathcal{V}_T$~(\ref{eq:function_variation})}.
\end{proof}

Proposition \ref{prop:vIQC} resembles the classical notion of IQCs, with the difference that (i), the integral quadratic term is additionally dependent on the time-variations of the minimizer \rev{and gradient}, and (ii), the right hand side of the constraint depends on the function value variation. \rev{Crucially, Proposition \ref{prop:vIQC} captures multiple sources of time-variation that inherently change the input-output behaviour of the operator.}
We denote (\ref{eq:IQC:delta}) as an instance of a \emph{variational IQC} (vIQC).

\section{Main results}
\label{sec:main}

\subsection{Regret with pointwise IQCs}
\label{sec:staticIQC}

It is instructive to start by showing regret bounds using pointwise IQCs. We assume the following assumption holds.
\begin{assumption}
    \label{assum:bounded_set}
    The feasible set $\mathcal{X}$ is bounded.
\end{assumption}

Consider (\ref{eq:generalized_algorithm:gradient}) and define for $i \in \mathbb{I}_p, j\in \mathbb{I}_q$ the filtered vectors
\begin{equation}
    \psi^i_t = \begin{bmatrix} L I_d & -I_d \\ -mI_d & I_d \end{bmatrix} \begin{bmatrix} s^i_t-x_t^* \\ \delta^i_t \end{bmatrix}, \quad \hat{\psi}^j_t = \begin{bmatrix} z^j_t - x_t^* \\ g^j_t \end{bmatrix}.
\end{equation}
Lemma~\ref{lemma:sector_IQC} and \ref{lemma:sector_IQC_h} imply that each $\psi^i_t,\, \hat{\psi}^j_t$ satisfy the quadratic inequalities $(\psi^i_t)^\top M_1 \psi_t^i \geq 0$ and $(\hat{\psi}^j_t)^\top M_1 \hat{\psi}^j_t \geq 0$, respectively. By stacking all $\psi_t^i$ and $\hat{\psi}_t^j$ into a vector $\psi_t$, and defining suitable matrices $D_\Psi^u$ and $D_\Psi^u$, we can write down the more compact relation
\begin{align*}
 \psi_t = 
 \begin{bmatrix}
    D_\Psi^y & D_\Psi^u
 \end{bmatrix}
 \begin{bmatrix}
    y_t - y_t^* \\ u_t
 \end{bmatrix},
\end{align*}
where the block rows of $D_\Psi^u$ and $D_\Psi^u$ select the respective components of $y_t - y_t^*$ and $u_t$ to realize $\psi_t^i$, $\hat{\psi}_t^j$. \rev{Moreover,} recall that $y_t^* = C \xi_t^*$, such that
\begin{align*}
    \psi_t &= \underbrace{\begin{bmatrix}
    D_\Psi^y C & D_\Psi^y D + D_\Psi^u
    \end{bmatrix}}_{=:  \begin{bmatrix} \hat{C} & \hat{D} \end{bmatrix}}
    \begin{bmatrix}
    \xi_t - \xi_t^* \\ u_t
    \end{bmatrix}.
\end{align*}
We state the following theorem.

\begin{theorem}
    \label{thm:static}
    Let Assumptions \ref{assum:fixed_point}, \ref{assum:Bu_Du} and \ref{assum:bounded_set} hold. If there exists a symmetric matrix $P \in \mathbb{S}^{n_\xi}$ and non-negative vectors $\lambda_p \in \mathbb{R}_{\geq 0}^p, \lambda_q \in \mathbb{R}_{\geq 0}^q$, such that $P\succ 0$ and the inequality
    \begin{multline}
        \label{eq:LMI:sector}
        \begin{bmatrix}
            A^\top P A - P & A^\top P B  \\
            B^\top P A & B^\top P B
        \end{bmatrix} 
        + \rev{\frac{1}{2}}
        \begin{bmatrix}
            0 & \bigl[ \begin{smallmatrix} C_1^\top & 0 \end{smallmatrix} \bigr] \\ \bigl[ \begin{smallmatrix} C_1 \\ 0 \end{smallmatrix} \bigr] & 0
        \end{bmatrix}
        \\ 
        +
        \begin{bmatrix}
            \hat{C} & \hat{D}
        \end{bmatrix}^\top
        M_\lambda
        \begin{bmatrix}
            \hat{C} & \hat{D}
        \end{bmatrix}
        \preceq 0
    \end{multline}
    holds with $M_\lambda = M_1 \otimes \mathrm{diag}(\lambda_p, \lambda_q)$, then we have
    \begin{equation}
        \label{eq:regret:sector:big_O}
        \mathcal{R}_T = \mathcal{O} \bigl(\lambda_{\max}(P) \mathcal{P}_T \bigr).
    \end{equation}
\end{theorem}

\begin{proof}
    We start by bounding the \rev{optimality difference in the $P$-normed state space}, that is
    \begin{multline}
        \label{proof:thm:static:1}
        \| \xi_{t+1} - \xi_{t+1}^* \|_P^2 = \| \xi_{t+1} - \xi_{t}^* + \xi_{t}^* - \xi_{t+1}^* \|_P^2 \\
        \qquad \qquad \,\, = \| \xi_{t+1} - \xi_{t}^* \|_P^2 + \| \xi_{t}^* - \xi_{t+1}^* \|_P^2 \\
        \qquad\qquad\qquad\qquad  + 2(\xi_{t+1} - \xi_{t}^*)^\top P (\xi_{t}^* - \xi_{t+1}^*) \\
        \qquad \leq \| \xi_{t+1} - \xi_{t}^* \|_P^2 + 3 \lambda_\mathrm{max}(P)R \| \xi_{t}^* - \xi_{t+1}^* \|,
    \end{multline}
    \rev{where $R$ is the diameter of the state space, which is finite by Assumption~\ref{assum:bounded_set}} and the observability of $(C_1,A)$.
    Next, we can bound the distance $\| \xi_{t+1} - \xi_{t}^* \|_P^2$ by
    \begin{multline*}
        \label{proof:thm:static:2}
        \begin{aligned}
            \| \xi_{t+1} - \xi_t^* \|_P^2 &= \| A (\xi_t- \xi_t^*) + B u_t  \|^2_P
        \end{aligned}
        \\
        = \begin{bmatrix}
        \xi_t - \xi_t^* \\ u_t 
        \end{bmatrix}^\top
        \begin{bmatrix}
            A^\top P A & A^\top P B  \\
            B^\top P A & B^\top P B
        \end{bmatrix} 
        \begin{bmatrix}
            \xi_t - \xi_t^* \\ u_t 
        \end{bmatrix} \\
        \stackrel{(\ref{eq:LMI:sector})}{\leq} \| \xi_{t} - \xi_t^* \|^2_P - (x_t - x_t^*)^\top \nabla f_t(x_t) \qquad \qquad \\ 
        - \sum_{i=1}^p \lambda_p^i (\psi_t^j)^\top M_1 \psi_t^j - \sum_{j=1}^q \lambda_q^j (\hat{\psi}_t^j)^\top M_1 \hat{\psi}_t^j.
    \end{multline*}
    The inequality follows by left and right multiplying (\ref{eq:LMI:sector}) with $\mathrm{vec}(\xi_t - \xi_t^*, u_t)$.
    We use Lemma \ref{lemma:sector_IQC} and \ref{lemma:sector_IQC_h}, and the fact that by convexity $(x_t - x_t^\star)^\top \nabla f_t(x_t) \geq f(x_t) - f(x_t^\star)$, to get
    \begin{equation}
        \label{proof:thm:static:3}
        \| \xi_{t+1} - \xi_t^* \|_P^2 \leq \| \xi_{t} - \xi_t^* \|^2_P - (f(x_t) - f(x_t^*)).
    \end{equation}
    Combining (\ref{proof:thm:static:1}) and (\ref{proof:thm:static:3}) yields
    \begin{multline}
        f(x_t) - f(x_t^*) \leq  \| \xi_{t} - \xi_t^* \|_P^2 - \| \xi_{t+1} - \xi_{t+1}^* \|_P^2 \\ 
        + 3 \lambda_{\max}(P) R \| \xi_{t}^* - \xi_{t+1}^* \| .
    \end{multline}
    Finally, summing from $t=1$ to $t=T$ we get
    \begin{multline}
        \label{eq:regret:sector:upper_bound}
        \sum_{t=1}^T f(x_t) - f(x_t^*) \leq \| \xi_{1} - \xi_1^* \|_P^2 - \| \xi_{T+1} - \xi_{T+1}^* \|_P^2 \\ 
        + 3 \lambda_{\max}(P) R \sum_{t=1}^T \| \xi_{t}^* - \xi_{t+1}^* \| .
    \end{multline}
    The left-hand side corresponds to $\mathcal{R}_T$ and, since $\xi^*_t = U x_t^*$, the last sum is $\mathcal{O}(\mathcal{P}_T)$. 
\end{proof}

Theorem~\ref{thm:static} shows that, given any algorithm satisfying the structural assumptions, proving a regret upper bound boils down to a feasibility problem in the form of a Linear Matrix Inequality (LMI). Therefore, Theorem~\ref{thm:static} \rev{offers} an \emph{automated} way to establish regret bounds, without the need for ad-hoc individual proofs. 

Qualitatively, (\ref{eq:regret:sector:upper_bound}) shows that the regret grows sublinearly in $T$ if the path length $\mathcal{P}_T$ does. 
\rev{Meanwhile, the sensitivity of the regret bound \emph{w.r.t.} the path length is determined by the maximum eigenvalue of the variable $P$, which can be optimized for by framing (\ref{eq:LMI:sector}) as the feasible set of an SDP whose objective it is to minimize the spectal radius of $P$.}


We \rev{investigate this path length sensitivity for} different algorithms with a numerical study\footnote{The source code for all numerical experiments can be accessed at: \href{https://github.com/col-tasas/2025-oco-with-iqcs}{https://github.com/col-tasas/2025-oco-with-iqcs}.}. In particular, we compare O-GD, Multi-step O-GD, O-NM, and O-AGD \cite{tinbao_yang}. We study the value of $\lambda_{\max}(P)$ for different convexity/Lipschitz moduli $m$ and $L$, and visualize the result over the function class condition ratio $\frac{L}{m}$. The results are shown in Fig.~\ref{fig:results_pointwise}. We observe that most algorithms attain a finite regret bound, with only O-NM \rev{growing unbounded} around $\frac{L}{m} \approx 8.5$. We observe how the factor $\lambda_{\max}(P)$ is proportional to $\frac{L}{m}$, indicating a fundamental dependence of (\ref{eq:regret:sector:big_O}) on the function class condition ratio. \rev{Note that detrimental effects of large $\frac{L}{m}$} are well known in \rev{static optimization}, but has not been particularly emphasized 
in the OCO field. Interestingly, we observe that, in terms of the upper bound (\ref{eq:regret:sector:big_O}), accelerated O-GD performs best among all compared algorithms, and also that 10-Step O-GD performs better than its counterparts with fewer steps.

\begin{figure}
    \centering
%
%
\definecolor{mycolor1}{RGB}{204, 51, 63} 
\definecolor{mycolor2}{RGB}{68, 12, 84}
\definecolor{mycolor3}{RGB}{53, 95, 141}
\definecolor{mycolor4}{RGB}{34, 168, 132}
\definecolor{mycolor5}{RGB}{122, 210, 81}

\begin{tikzpicture}

\begin{axis}[%
width=2.5in,
height=1in,
at={(0.758in,0.481in)},
scale only axis,
xmode=log,
xmin=1,
xmax=100,
xminorticks=true,
xlabel style={font=\color{white!15!black}},
xlabel={$\frac{L}{m}$},
ymin=0,
ymax=500,
ylabel={$\lambda_{\max}(P)$},
ylabel style={font=\color{white!15!black}},
axis background/.style={fill=white},
xmajorgrids,
xminorgrids,
ymajorgrids,
legend style={at={(0.015,0.97)}, nodes={scale=0.65, transform shape}, anchor=north west, legend cell align=left, align=left, draw=white!15!black}
]
\addplot [color=mycolor1, line width=1.5pt, solid]
  table[row sep=crcr]{%
1	2.00000033229368\\
1.09854114198756	2.20193749511947\\
1.20679264063933	2.43496690771006\\
1.32571136559011	2.70446673547321\\
1.45634847750124	3.01682392641439\\
1.59985871960606	3.37963269797932\\
1.75751062485479	3.80193250820842\\
1.93069772888325	4.29449463146938\\
2.12095088792019	4.87016723224842\\
2.32995181051537	5.54428953526404\\
2.55954792269954	6.33519070830222\\
2.81176869797423	7.26479028318547\\
3.08884359647748	8.35932083568708\\
3.39322177189533	9.65019866501293\\
3.72759372031494	11.1750711498916\\
4.09491506238042	12.9790797065158\\
4.49843266896945	15.1163807919553\\
4.94171336132383	17.65197879913\\
5.42867543932386	20.6639335955968\\
5.96362331659464	24.2460245627693\\
6.55128556859551	28.5109563488646\\
7.19685673001152	33.5942300261278\\
7.9060432109077	39.6588023481321\\
8.68511373751353	46.9007073989776\\
9.54095476349994	55.5558613838223\\
10.4811313415469	65.9081880284423\\
11.5139539932645	78.2995192842797\\
12.648552168553	93.141473046014\\
13.8949549437314	110.929841140934\\
15.2641796717523	132.261747692375\\
16.7683293681101	157.85675768348\\
18.4206996932672	188.581777254966\\
20.2358964772516	225.48163708172\\
22.2299648252619	269.815622835709\\
24.4205309454865	323.101661000925\\
26.8269579527972	387.169793155443\\
29.4705170255181	464.22620141671\\
32.3745754281764	556.931139086758\\
35.5648030622313	668.492394468353\\
39.0693993705462	802.778380277979\\
42.9193426012878	964.454322954933\\
47.1486636345739	1159.14690080499\\
51.7947467923121	1393.64263501333\\
56.898660290183	1676.12740734246\\
62.5055192527397	2016.4753034607\\
68.66488450043	2426.59799945596\\
75.4312006335462	2920.86339762289\\
82.8642772854684	3516.60846202581\\
91.0298177991522	4234.74365144312\\
100	5100.49948943705\\
};
\addlegendentry{O-GD}

\addplot [color=mycolor4, line width=1.5pt, dashed]
  table[row sep=crcr]{%
1	2.00000196161398\\
1.09854114198756	2.19709297206221\\
1.20679264063933	2.41377140385525\\
1.32571136559011	2.65244309552431\\
1.45634847750124	2.91617105465786\\
1.59985871960606	3.20881115083164\\
1.75751062485479	3.5351536140186\\
1.93069772888325	3.90107209528099\\
2.12095088792019	4.31368906724514\\
2.32995181051537	4.78156763282783\\
2.55954792269954	5.31494149608198\\
2.81176869797423	5.92599596953456\\
3.08884359647748	6.62921416730556\\
3.39322177189533	7.44179940131758\\
3.72759372031494	8.38419489209547\\
4.09491506238042	9.48071695317412\\
4.49843266896945	10.7603154825388\\
4.94171336132383	12.25750761121\\
5.42867543932386	14.0134844520809\\
5.96362331659464	16.0774744227884\\
6.55128556859551	18.5083509133188\\
7.19685673001152	21.376601263208\\
7.9060432109077	24.7666755896794\\
8.68511373751353	28.7798282628181\\
9.54095476349994	33.5375411044916\\
10.4811313415469	39.185637749044\\
11.5139539932645	45.899246588681\\
12.648552168553	53.8888067201728\\
13.8949549437314	63.4072438641187\\
15.2641796717523	74.758700889988\\
16.7683293681101	88.3089999222611\\
18.4206996932672	104.498299693721\\
20.2358964772516	123.856336033479\\
22.2299648252619	147.020765698097\\
24.4205309454865	174.759429606324\\
26.8269579527972	207.996975450798\\
29.4705170255181	247.847197345011\\
32.3745754281764	295.65189903763\\
35.5648030622313	353.02781089561\\
39.0693993705462	421.9232091911\\
42.9193426012878	504.686243713335\\
47.1486636345739	604.147324247811\\
51.7947467923121	723.718238004086\\
56.898660290183	867.512655311365\\
62.5055192527397	1040.49000174344\\
68.66488450043	1248.6310656823\\
75.4312006335462	1499.14723217587\\
82.8642772854684	1800.73594865258\\
91.0298177991522	2163.88615501179\\
100	2601.24986403606\\
};
\addlegendentry{2-Step O-GD}

\addplot [color=mycolor5, line width=1.5pt, dotted]
  table[row sep=crcr]{%
1	2.00000237466185\\
1.09854114198756	2.19708236106007\\
1.20679264063933	2.41358529512425\\
1.32571136559011	2.651422793219\\
1.45634847750124	2.9126970393271\\
1.59985871960606	3.19971751813473\\
1.75751062485479	3.51502126741014\\
1.93069772888325	3.86139548255549\\
2.12095088792019	4.24190181038972\\
2.32995181051537	4.65990369451666\\
2.55954792269954	5.11909622174353\\
2.81176869797423	5.623539344076\\
3.08884359647748	6.17769624744866\\
3.39322177189533	6.78647948806193\\
3.72759372031494	7.45531195152819\\
4.09491506238042	8.1902134008799\\
4.49843266896945	8.99792897666669\\
4.94171336132383	9.88612111682863\\
5.42867543932386	10.8636460768722\\
5.96362331659464	11.9409333069822\\
6.55128556859551	13.1304791881302\\
7.19685673001152	14.4474499673626\\
7.9060432109077	15.9103899014152\\
8.68511373751353	17.5420141390082\\
9.54095476349994	19.3700823379125\\
10.4811313415469	21.4283755652703\\
11.5139539932645	23.7577599714954\\
12.648552168553	26.4073798617341\\
13.8949549437314	29.4361481780014\\
15.2641796717523	32.9143959852115\\
16.7683293681101	36.9259840621355\\
18.4206996932672	41.5707401540955\\
20.2358964772516	46.9676078912858\\
22.2299648252619	53.2582827122111\\
24.4205309454865	60.6116434464993\\
26.8269579527972	69.2293463866985\\
29.4705170255181	79.3522770000381\\
32.3745754281764	91.2685078650741\\
35.5648030622313	105.322723441805\\
39.0693993705462	121.927969272188\\
42.9193426012878	141.579356545965\\
47.1486636345739	164.870745647475\\
51.7947467923121	192.514959432475\\
56.898660290183	225.367809776881\\
62.5055192527397	264.457861723233\\
68.66488450043	311.021152159688\\
75.4312006335462	366.543752376431\\
82.8642772854684	432.813187222365\\
91.0298177991522	511.979733910644\\
100	606.632358254733\\
};
\addlegendentry{10-Step O-GD}

\addplot [color=mycolor3, line width=1.5pt, dashdotted]
  table[row sep=crcr]{%
1	2.0000018234507\\
1.09854114198756	2.30397485840304\\
1.20679264063933	2.65649650310947\\
1.32571136559011	3.0650577102961\\
1.45634847750124	3.53827806941351\\
1.59985871960606	4.08609299854901\\
1.75751062485479	4.71998938135094\\
1.93069772888325	5.45330904100571\\
2.12095088792019	6.30165760809369\\
2.32995181051537	7.28349121910101\\
2.55954792269954	8.42103049793493\\
2.81176869797423	9.74185989072046\\
3.08884359647748	11.2821635998749\\
3.39322177189533	13.0945908101095\\
3.72759372031494	15.2705920980032\\
4.09491506238042	17.9954828050103\\
4.49843266896945	21.6116402600581\\
4.94171336132383	26.6897213855695\\
5.42867543932386	34.2998270105008\\
5.96362331659464	46.8424212222535\\
6.55128556859551	71.1540139653898\\
7.19685673001152	137.673770729665\\
7.9060432109077	1047.15584390362\\
};
\addlegendentry{O-NM}

\addplot [color=mycolor2, line width=1.5pt, densely dotted]
  table[row sep=crcr]{%
  1	2.00000205654993\\
  1.09854114198756	1.67825425213788\\
  1.20679264063933	1.45643594548896\\
  1.32571136559011	1.36725884331788\\
  1.45634847750124	1.46369278869633\\
  1.59985871960606	1.70506449150517\\
  1.75751062485479	2.02963880970807\\
  1.93069772888325	2.41247318630977\\
  2.12095088792019	2.84553145341732\\
  2.32995181051537	3.32653527608361\\
  2.55954792269954	3.85525073111267\\
  2.81176869797423	4.43283618917956\\
  3.08884359647748	5.06185638566962\\
  3.39322177189533	5.74555061688335\\
  3.72759372031494	6.48797398376796\\
  4.09491506238042	7.29404824153815\\
  4.49843266896945	8.16955888618118\\
  4.94171336132383	9.12116631410462\\
  5.42867543932386	10.1564205246961\\
  5.96362331659464	11.2837897617109\\
  6.55128556859551	12.512704687927\\
  7.19685673001152	13.8536097434356\\
  7.9060432109077	15.3180328006456\\
  8.68511373751353	16.9186685322\\
  9.54095476349994	18.6694723002518\\
  10.4811313415469	20.5857726031559\\
  11.5139539932645	22.6843911263505\\
  12.648552168553	24.9838020345676\\
  13.8949549437314	27.5042560222693\\
  15.2641796717523	30.2679836080495\\
  16.7683293681101	33.2993792098095\\
  18.4206996932672	36.6252040669656\\
  20.2358964772516	40.2748414999661\\
  22.2299648252619	44.2805322701254\\
  24.4205309454865	48.6776706083159\\
  26.8269579527972	53.505121101727\\
  29.4705170255181	58.8055342796291\\
  32.3745754281764	64.6257724190735\\
  35.5648030622313	71.0172545836642\\
  39.0693993705462	78.036497149108\\
  42.9193426012878	85.7455517151924\\
  47.1486636345739	94.2125327863671\\
  51.7947467923121	103.512307122727\\
  56.898660290183	113.727021831407\\
  62.5055192527397	124.94706729218\\
  68.66488450043	137.271528204716\\
  75.4312006335462	150.809373238339\\
  82.8642772854684	165.680271428442\\
  91.0298177991522	182.015693773972\\
  100	199.959999552068\\
};
\addlegendentry{O-AGD}

\end{axis}
\end{tikzpicture}%
    \caption{\rev{Regret sensitivity to $\mathcal{P}_T$ over function class condition ratio, according to Theorem \ref{thm:static}.}}
    \label{fig:results_pointwise}
\end{figure}

\subsection{Regret with variational IQCs}
\label{sec:vIQC}

\rev{The use of pointwise IQC for the characterization of slope-restricted nonlinearities is a known source of conservatism \cite{lessard}. We therefore now develop a regret bound that leverages the variational IQCs.}

We introduce the fixed-point variation $\Delta \xi_t^\star := \xi_{t}^\star - \xi_{t+1}^\star$. Using this, we can write the evolution of (\ref{eq:generalized_algorithm}) in error coordinates $\tilde{\xi}_t := \xi_t - \xi_t^\star$ as
\begin{align}
    \begin{aligned}
        \label{eq:generalized_algorithm_delta}
    \tilde{\xi}_{t+1} &= A \tilde{\xi}_t + B u_t + \Delta \xi^\star_t \\
    \tilde{y}_t &= C \tilde{\xi}_t + D u_t.
    \end{aligned}
\end{align}



\rev{
We now apply Proposition~\ref{prop:vIQC} to each subcomponent of $(s_t, \delta_t)$. We moreover apply Proposition~\ref{prop:vIQC} to $(z_t, g_t)$ by letting $f_t \leftarrow \mathcal{I}_{\mathcal{X}}$, $\nabla f_t \leftarrow \mathcal{N}_\mathcal{X}$ and $m \rightarrow 0$ and $L \rightarrow \infty$. Since $\mathcal{I}_\mathcal{X}$ is not time-varying, the gradient and function variation is not needed, and the vIQC simplifies drastically. 
That is, we can define
\begin{align}
    \psi_t^i = \Psi_f \begin{bmatrix} s^i_t - x_t^\star \\ \delta_t^i \\ \Delta x_t^\star \\ \Delta \delta_t(s^i_t) \end{bmatrix}, \quad \hat{\psi}_t^j = \Psi_{\mathcal{I}} \begin{bmatrix} z_t^j - x_t^\star \\ g_t^j \\ \Delta x_t^\star \end{bmatrix} 
\end{align}
with $\Psi_{f}$ and $\Delta \delta_t$ as defined in Proposition~\ref{prop:vIQC}} and $\Psi_{\mathcal{I}}$ adapted accordingly, and it holds $\sum_{t=1}^{T} (\psi^i_t)^\top M_2 \psi_t^i \geq -4(L-m)\mathcal{V}_T$ and $\sum_{t=1}^{T} (\hat{\psi}^j_t)^\top M_1 \hat{\psi}_t^j \geq 0$ for all $i \in \mathbb{I}_p$, $j \in \mathbb{I}_q$. 
\ifthenelse{\boolean{arxiv}}{
    We refer to Appendix \ref{appendix:A} for the discussion on $\Psi_{\mathcal{I}}$.
}{
    We refer to the extended version of this work \cite[Appendix A]{jakob_iannelli_b} for the discussion on $\Psi_{\mathcal{I}}$.
}%
 

Analogously to the last section, we stack all $\psi_t^i$ and $\hat{\psi}_t^j$ into a vector $\psi_t$, to get the a compact vIQC
\rev{
\begin{align*}
     \psi_t = 
     \left[ \begin{array}{c|cccc}
        A_\Psi & B_\psi^y & B_\Psi^u & B_\Psi^{\Delta \xi} & B_\Psi^{\Delta \delta} \\ \hline
        C_\Psi & D_\psi^y & D_\Psi^u & D_\Psi^{\Delta \xi} & D_\Psi^{\Delta \delta}
        \end{array} \right]
        \begin{bmatrix}
        \tilde{y}_t \\ u_t \\ \Delta \xi_t^\star \\ \bm{\Delta \delta}_t
        \end{bmatrix}
\end{align*}
where we defined $\bm{\Delta \delta}_t := \mathrm{vec}(\Delta \delta_t(s^1_t),\dots, \Delta \delta_t(s^p_t))$. The respective matrices result from a straightforward rearrangement and stacking of the filter matrices. In particular, we leveraged that $(1_{p+q} \otimes I_d) \Delta x_t^\star = C \Delta \xi_t^\star$.}
\rev{Together with (\ref{eq:generalized_algorithm_delta}), we can build the augmented plant
{\setlength{\arraycolsep}{3pt}
\begin{align}\label{eq:augmented_plant}
    \left[ \begin{array}{c|c c c} \hat{A} & \hat{B}_u & \hat{B}_{\Delta \xi} & \hat{B}_{\Delta \delta} \\ \hline \hat{C} & \hat{D}_u & \hat{D}_{\Delta \xi} & \hat{D}_{\Delta \delta} \end{array} \right]
    \triangleq
    \left[ \begin{array}{cc|ccc} A & 0 & B & I & 0 \\ B_\Psi^y C & A_\Psi & B_\Psi^u & B_\Psi^{\Delta \xi} & B_\Psi^{\Delta \delta} \\ \hline D_\Psi^y C & C_\Psi & D_\Psi^u & D_\Psi^{\Delta \xi} & D_\Psi^{\Delta \delta} \end{array} \right],
\end{align}
}
defining the mapping $(u_t, \Delta \xi_t^\star, \bm{\Delta \delta}_t) \mapsto \psi_t$.
}

\begin{theorem}
    \label{thm:dynamic}
    Let Assumptions \ref{assum:fixed_point} and \ref{assum:Bu_Du} hold. Consider the augmented plant (\ref{eq:augmented_plant}).
    If there exists a symmetric matrix $P \in \mathbb{S}^{n_\xi+n_\zeta}$, non-negative vectors $\lambda_p \in \mathbb{R}_{\geq 0}^p, \lambda_q \in \mathbb{R}_{\geq 0}^q$ and scalars \rev{$\gamma_{\Delta \xi}, \gamma_{\Delta \delta} > 0$}, such that $P\succ 0$ and the LMI
    {\setlength{\arraycolsep}{3pt}
    \rev{
    \begin{multline}
        \label{eq:LMI:variational}
        \hspace{-10pt}
        \begin{bmatrix}
            \star
        \end{bmatrix}^\top
        \hspace{-3pt}
        \begin{bmatrix}
            -P &  &  & & \\
            & P & & & \\
            & & M_\lambda & & \\
            & & & -\gamma_{\Delta \xi} I & \\
            & & & &  -\gamma_{\Delta \delta} I
        \end{bmatrix}
        \hspace{-3pt}
        \begin{bmatrix}
        I & 0 & 0 & 0\\
        \hat{A} & \hat{B}_u & \hat{B}_{\Delta \xi} & \hat{B}_{\Delta \delta} \\ \hat{C} & \hat{D}_u & \hat{D}_{\Delta \xi} & \hat{D}_{\Delta \delta} \\
        0 & 0 & I & 0 \\ 0 & 0 & 0 & I
        \end{bmatrix}
        \\ + \frac{1}{2} 
        \begin{bmatrix}
            0 & \bigl[ \begin{smallmatrix} C_1^\top & 0 \\ 0 & 0 \end{smallmatrix} \bigr] \\ \bigl[ \begin{smallmatrix} C_1 & 0 \\ 0 & 0 \end{smallmatrix} \bigr] & 0
        \end{bmatrix}
        \preceq 0 
    \end{multline}
    }
    }
    holds with $M_\lambda = \mathrm{blkdiag}(M_2 \otimes \mathrm{diag}(\lambda_p), M_1 \otimes \mathrm{diag}(\lambda_q))$, then we have
    \rev{
    \begin{equation}
        \label{eq:regret:dynamic:bigO}
        \mathcal{R}_T = \mathcal{O}\bigl(\gamma_1 \mathcal{S}_T + \gamma_2 \mathcal{G}_T + \gamma_3 \mathcal{V}_T \bigr),
    \end{equation}
    with $\gamma_1 = \gamma_{\Delta \xi}$, $\,\gamma_2 = p \gamma_{\Delta \delta}$, $\,\gamma_3 = (L-m)\sum_{i=1}^p \lambda_p^i$.
    }
\end{theorem}
\begin{proof}
    Define $\eta_t$ as the state of (\ref{eq:augmented_plant}). Left and right multiply (\ref{eq:LMI:variational}) by $\mathrm{vec}(\eta_t, u_t, \Delta \xi^\star_t, \rev{\bm{\Delta \delta}_t})$, to obtain the inequality
    \begin{multline*}
        -\| \eta_t \|^2_P + \| \eta_{t+1} \|^2_P - \gamma_{\Delta \xi} \| \Delta \xi_t^\star \|^2 \rev{- \gamma_{\Delta \delta} \sum_{i=1}^p \| \Delta \delta_t(s_t^i) \|^2} \\
        + \sum_{i=1}^p \lambda_p^i (\psi_t^j)^\top M_2 \psi_t^j + \sum_{j=1}^q \lambda_q^j (\hat{\psi}_t^j)^\top M_1 \hat{\psi}_t^j \\ + (x_t - x_t^\star)^\top \nabla f_t(x_t) \leq 0
    \end{multline*}
    \rev{Note that $\sum_{t=1}^T \sum_{i=1}^{p} \| \Delta \delta_t(s_t^i) \|^2 \leq p \mathcal{G}_T$.}
    Leveraging again convexity and Proposition \ref{prop:vIQC}, and summing from $t=1$ to $T$ gives, after telescoping and rearranging terms, that
    \rev{
    \begin{equation}
        \label{eq:regret:dynamic:upper_bound}
        \mathcal{R}_T \leq \| \eta_{1} \|^2_P - \|\eta_{T+1} \|^2_P + \gamma_1 \sum_{t=1}^{T} \| \Delta \xi_t^\star \|^2 +  \gamma_2 \mathcal{G}_T + \gamma_3 \mathcal{V}_T.
    \end{equation}
    }
    By $\xi_t^\star = U x_t^\star$, the sum is $\mathcal{O}(\mathcal{S}_T)$, which gives the result.
\end{proof}

In contrast to (\ref{eq:regret:sector:big_O}), the bound in (\ref{eq:regret:dynamic:bigO}) depends on the squared path length, \rev{gradient variation}, and function variation, with sensitivities $\gamma_1$, $\gamma_2$, and $\gamma_3$ that \rev{are decision variables in the LMI (\ref{eq:LMI:variational}). These degrees of freedom can be optimized to tighten the upper bound (\ref{eq:regret:dynamic:upper_bound}), for example by minimizing a weighted sum $k_1 \gamma_1 + k_2 \gamma_2 + k_3 \gamma_3$ for some $k_1, k_2, k_3 > 0$.} However, the trade-off between those values depends on problem-specific considerations, i.e. on the foresight on $\mathcal{S}_T$, \rev{$\mathcal{G}_T$}, and $\mathcal{V}_T$. I.e., setting \rev{$k_1 = \mathcal{S}_T$, $k_2 = \mathcal{G}_T$, $k_3 = \mathcal{V}_T$} may be beneficial. 


Fig.~\ref{fig:results_variational} \rev{presents a numerical study investigating the algorithm and function class dependence of these sensitivites, using the same list of algorithms as in Fig.~\ref{fig:results_pointwise}}. We observe analogously to Fig.~\ref{fig:results_pointwise} a deteriorating effect of large condition ratios $\frac{L}{m}$. \rev{Note especially that the relative performance of algorithms differs \emph{w.r.t.} which variation metric is considered. As an example, 2-step O-GD tends to be favourable when $\mathcal{G}_T$ is expected to dominate (because it achieves consistently the lowest value of $\gamma_2$), but may perform worse when $\mathcal{S}_T$ is the primary source of variation. Thus, this worst-case bound may provide guidance on the algorithm choice when prior information is available.}
Moreover, Theorem~\ref{thm:dynamic} extends the regret bound for O-NM to higher condition ratios, \rev{indicating a finite regret bound over the whole function class}. We observe that the bounds from both theorems offer distinct information, highlighting the complementary nature of both.

\begin{figure}
    \begin{subfigure}{0.9\columnwidth}
%
%
\definecolor{mycolor1}{RGB}{204, 51, 63} 
\definecolor{mycolor2}{RGB}{68, 12, 84}
\definecolor{mycolor3}{RGB}{53, 95, 141}
\definecolor{mycolor4}{RGB}{34, 168, 132}
\definecolor{mycolor5}{RGB}{122, 210, 81}

\begin{tikzpicture}

\begin{axis}[%
width=2.5in,
height=1in,
at={(0.758in,0.481in)},
scale only axis,
xmode=log,
xmin=1,
xmax=100,
xminorticks=true,
xlabel style={font=\color{white!15!black}},
ymin=0,
ymax=3000,
ylabel={$\gamma_1$},
ylabel style={font=\color{white!15!black}, yshift=-0.00in},
axis background/.style={fill=white},
xmajorgrids,
xminorgrids,
ymajorgrids,
legend style={at={(0.015,0.97)}, nodes={scale=0.7, transform shape}, anchor=north west, legend cell align=left, align=left, draw=white!15!black}
]
\addplot [color=mycolor1, line width=1.5pt, solid]
  table[row sep=crcr]{%
1.0    1.99947665\\
1.1450475699382818    2.65314586\\
1.3111339374215643    3.5448592\\
1.5013107289081733    4.77114748\\
1.7190722018585745    6.4694991\\
1.9684194472866121    8.83769643\\
2.2539339047347906    12.16147722\\
2.580861540418075    14.20322893\\
2.955209235202887    13.82941323\\
3.383855153428234    12.38567485\\
3.874675120456131    12.1301645\\
4.436687330978613    13.26252497\\
5.0802180469130205    15.56638708\\
5.817091329374358    19.59062738\\
6.6608462908091575    25.07595098\\
7.626985859023444    32.66229512\\
8.733261623828433    43.10261493\\
10.0    57.47274445\\
11.450475699382817    77.2767011\\
13.111339374215643    104.62557301\\
15.013107289081734    142.46516468\\
17.190722018585745    194.9130021\\
19.68419447286612    267.73089092\\
22.539339047347912    368.96122216\\
25.808615404180742    509.9080077\\
29.552092352028865    706.46103355\\
33.83855153428233    980.97568335\\
38.7467512045613    1365.11194711\\
44.366873309786115    1903.97311232\\
50.8021804691302    2662.19927419\\
58.170913293743574    3732.53319704\\
66.60846290809158    5248.34050686\\
76.26985859023443    7399.42411529\\
87.33261623828433    10461.49327992\\
100.0    14813.99594152\\
};
\addlegendentry{O-GD}

\addplot [color=mycolor4, line width=1.5pt, dashed]
table[row sep=crcr]{%
1.0    1.99939999\\
1.1450475699382818    2.31245449\\
1.3111339374215643    2.72573494\\
1.5013107289081733    3.27383842\\
1.7190722018585745    4.00389839\\
1.9684194472866121    4.98167168\\
2.2539339047347906    6.29951388\\
2.580861540418075    8.08793102\\
2.955209235202887    10.53241383\\
3.383855153428234    13.89267623\\
3.874675120456131    16.61027247\\
4.436687330978613    22.07996592\\
5.0802180469130205    30.00829891\\
5.817091329374358    41.44510853\\
6.6608462908091575    58.05773318\\
7.626985859023444    82.31078222\\
8.733261623828433    117.84395824\\
10.0    170.00375718\\
11.450475699382817    246.66523461\\
13.111339374215643    359.45536089\\
15.013107289081734    525.59888467\\
17.190722018585745    770.67957105\\
19.68419447286612    1132.74370487\\
22.539339047347912    1668.44311044\\
25.808615404180742    2462.19046846\\
29.552092352028865    3639.90041822\\
33.83855153428233    5389.4878794\\
38.7467512045613    7991.72205936\\
44.366873309786115    11866.01754434\\
50.8021804691302    17639.15349618\\
58.170913293743574    26251.02517902\\
66.60846290809158    39105.07024012\\
76.26985859023443    58309.65046861\\
87.33261623828433    87003.93784111\\
};
\addlegendentry{2-step O-GD}

\addplot [color=mycolor5, line width=1.5pt, dotted]
table[row sep=crcr]{%
1.0    1.99846595\\
1.1450475699382818    2.29022146\\
1.3111339374215643    2.62255298\\
1.5013107289081733    3.00286274\\
1.7190722018585745    3.43852112\\
1.9684194472866121    3.93707659\\
2.2539339047347906    4.50858502\\
2.580861540418075    5.16481992\\
2.955209235202887    5.92136351\\
3.383855153428234    6.80026545\\
3.874675120456131    7.83356467\\
4.436687330978613    9.0679131\\
5.0802180469130205    10.57028927\\
5.817091329374358    12.43510929\\
6.6608462908091575    14.79367257\\
7.626985859023444    17.63239349\\
8.733261623828433    19.36096895\\
10.0    22.67782949\\
11.450475699382817    27.76093079\\
13.111339374215643    34.92871646\\
15.013107289081734    44.91744253\\
17.190722018585745    58.9830807\\
19.68419447286612    78.98571744\\
22.539339047347912    107.65018807\\
25.808615404180742    148.99761666\\
29.552092352028865    209.0191681\\
33.83855153428233    296.49588861\\
38.7467512045613    424.32583259\\
44.366873309786115    611.39124706\\
50.8021804691302    885.65340907\\
58.170913293743574    1287.67518103\\
66.60846290809158    1878.91372658\\
76.26985859023443    2750.82291741\\
87.33261623828433    4033.30848684\\
100.0    5929.73977365\\
};
\addlegendentry{10-step O-GD}

\addplot [color=mycolor3, line width=1.5pt, dashdotted]
table[row sep=crcr]{%
1.0    1.99970916\\
1.1450475699382818    2.68567883\\
1.3111339374215643    3.64049364\\
1.5013107289081733    4.65347666\\
1.7190722018585745    5.69805107\\
1.9684194472866121    6.45473917\\
2.2539339047347906    6.94758426\\
2.580861540418075    7.61646264\\
2.955209235202887    9.20712374\\
3.383855153428234    11.66736569\\
3.874675120456131    15.01909469\\
4.436687330978613    19.4497539\\
5.0802180469130205    25.36631931\\
5.817091329374358    33.3944892\\
6.6608462908091575    44.43148706\\
7.626985859023444    59.78379942\\
8.733261623828433    81.41104871\\
10.0    112.319185\\
11.450475699382817    157.2083986\\
13.111339374215643    223.39332106\\
15.013107289081734    321.71903907\\
17.190722018585745    466.42475625\\
19.68419447286612    672.65170715\\
22.539339047347912    954.50060415\\
25.808615404180742    1330.01345556\\
29.552092352028865    1829.8605421\\
33.83855153428233    2498.96571042\\
38.7467512045613    3397.59446983\\
44.366873309786115    4605.44515066\\
50.8021804691302    6228.87858279\\
58.170913293743574    8409.04057355\\
66.60846290809158    11334.34982501\\
76.26985859023443    15255.38295791\\
87.33261623828433    20505.41367909\\
100.0    27526.94306997\\
};
\addlegendentry{O-NM}

\addplot [color=mycolor2, line width=1.5pt, densely dotted]
table[row sep=crcr]{%
1.0    1.99826983\\
1.1450475699382818    2.65161985\\
1.3111339374215643    3.54014475\\
1.5013107289081733    4.35721737\\
1.7190722018585745    5.36218367\\
1.9684194472866121    6.22323323\\
2.2539339047347906    6.82573307\\
2.580861540418075    7.00192763\\
2.955209235202887    6.52071045\\
3.383855153428234    6.00179734\\
3.874675120456131    5.97061874\\
4.436687330978613    6.38646524\\
5.0802180469130205    7.22474683\\
5.817091329374358    8.65543403\\
6.6608462908091575    10.65362295\\
7.626985859023444    13.48421716\\
8.733261623828433    17.59666075\\
10.0    23.73005398\\
11.450475699382817    33.09594823\\
13.111339374215643    47.70138038\\
15.013107289081734    70.94004377\\
17.190722018585745    108.61339145\\
19.68419447286612    170.38190636\\
22.539339047347912    271.61961448\\
25.808615404180742    434.99575137\\
29.552092352028865    688.35205633\\
33.83855153428233    1061.04241681\\
38.7467512045613    1587.52975403\\
44.366873309786115    2314.1212332\\
50.8021804691302    3304.22007979\\
58.170913293743574    4642.62621625\\
66.60846290809158    6443.17803481\\
76.26985859023443    8855.15225756\\
87.33261623828433    12075.45323483\\
100.0    16365.75778593\\
};
\addlegendentry{O-AGD}

\end{axis}
\end{tikzpicture}%
    \end{subfigure}
    \hfill
    \begin{subfigure}{0.9\columnwidth}
%
%
\definecolor{mycolor1}{RGB}{204, 51, 63} 
\definecolor{mycolor2}{RGB}{68, 12, 84}
\definecolor{mycolor3}{RGB}{53, 95, 141}
\definecolor{mycolor4}{RGB}{34, 168, 132}
\definecolor{mycolor5}{RGB}{122, 210, 81}

\begin{tikzpicture}

\begin{axis}[%
width=2.5in,
height=1in,
at={(0.758in,0.481in)},
scale only axis,
xmode=log,
xmin=1,
xmax=100,
xminorticks=true,
xlabel style={font=\color{white!15!black}},
ymin=0,
ymax=100,
ylabel={$\gamma_2$},
ylabel style={font=\color{white!15!black}, yshift=0.11in},
axis background/.style={fill=white},
xmajorgrids,
xminorgrids,
ymajorgrids,
legend style={at={(0.015,0.97)}, nodes={scale=0.7, transform shape}, anchor=north west, legend cell align=left, align=left, draw=white!15!black}
]
\addplot [color=mycolor1, line width=1.5pt, solid]
table[row sep=crcr]{%
1.0    0.00054815\\
1.1450475699382818    0.00000110966127\\
1.3111339374215643    0.00000105650348\\
1.5013107289081733    0.00000110965566\\
1.7190722018585745    0.00000101482743\\
1.9684194472866121    0.00000106887937\\
2.2539339047347906    0.0000185433332\\
2.580861540418075    1.5279679\\
2.955209235202887    4.62203663\\
3.383855153428234    8.0748965\\
3.874675120456131    11.13855474\\
4.436687330978613    14.06104287\\
5.0802180469130205    17.17310743\\
5.817091329374358    20.58235434\\
6.6608462908091575    24.72121553\\
7.626985859023444    29.7288176\\
8.733261623828433    35.83751342\\
10.0    43.32661903\\
11.450475699382817    52.54565197\\
13.111339374215643    63.92749701\\
15.013107289081734    78.01335341\\
17.190722018585745    95.48083695\\
19.68419447286612    117.1817027\\
22.539339047347912    144.17390247\\
25.808615404180742    177.77610124\\
29.552092352028865    219.64198925\\
33.83855153428233    271.8181954\\
38.7467512045613    336.84260046\\
44.366873309786115    417.85777844\\
50.8021804691302    518.69132042\\
58.170913293743574    644.11243199\\
66.60846290809158    800.06598058\\
76.26985859023443    994.52757161\\
87.33261623828433    1236.04961323\\
100.0    1539.23804253\\
};
\addlegendentry{O-GD}

\addplot [color=mycolor4, line width=1.5pt, dashed]
table[row sep=crcr]{%
1.0    0.00062223\\
1.1450475699382818    0.00000471991392\\
1.3111339374215643    0.00000528965726\\
1.5013107289081733    0.00000589210175\\
1.7190722018585745    0.00000722562596\\
1.9684194472866121    0.00000947609727\\
2.2539339047347906    0.0000135729204\\
2.580861540418075    0.0000236096911\\
2.955209235202887    0.0000530714057\\
3.383855153428234    0.0034316\\
3.874675120456131    1.00418519\\
4.436687330978613    1.31323346\\
5.0802180469130205    1.60014433\\
5.817091329374358    1.89801348\\
6.6608462908091575    2.19705233\\
7.626985859023444    2.48900357\\
8.733261623828433    2.76021869\\
10.0    2.99821375\\
11.450475699382817    3.19546054\\
13.111339374215643    3.35253882\\
15.013107289081734    3.47659928\\
17.190722018585745    3.57565554\\
19.68419447286612    3.65477996\\
22.539339047347912    3.71808209\\
25.808615404180742    3.77393481\\
29.552092352028865    3.81503992\\
33.83855153428233    3.84578729\\
38.7467512045613    3.86967327\\
44.366873309786115    3.89676451\\
50.8021804691302    3.8603451\\
58.170913293743574    3.61311997\\
66.60846290809158    3.16243631\\
76.26985859023443    3.04837136\\
87.33261623828433    3.64377855\\
};
\addlegendentry{2-step O-GD}

\addplot [color=mycolor5, line width=1.5pt, dotted]
table[row sep=crcr]{%
1.0    0.00157125\\
1.1450475699382818    0.0000176543582\\
1.3111339374215643    0.0000337661152\\
1.5013107289081733    0.0000297860495\\
1.7190722018585745    0.0000460124113\\
1.9684194472866121    0.0000210995791\\
2.2539339047347906    0.0000257519511\\
2.580861540418075    0.0000291335687\\
2.955209235202887    0.0000301052019\\
3.383855153428234    0.0000365304296\\
3.874675120456131    0.0000460140629\\
4.436687330978613    0.0000573359377\\
5.0802180469130205    0.0000795634668\\
5.817091329374358    0.00013119\\
6.6608462908091575    0.00026011\\
7.626985859023444    0.12396294\\
8.733261623828433    1.28878908\\
10.0    1.95226365\\
11.450475699382817    2.36888917\\
13.111339374215643    2.73951535\\
15.013107289081734    3.14539356\\
17.190722018585745    3.59245257\\
19.68419447286612    4.06852121\\
22.539339047347912    4.57122746\\
25.808615404180742    5.09382913\\
29.552092352028865    5.63819465\\
33.83855153428233    6.16919902\\
38.7467512045613    6.6547005\\
44.366873309786115    7.07540744\\
50.8021804691302    7.4196252\\
58.170913293743574    7.7227487\\
66.60846290809158    7.93108671\\
76.26985859023443    8.11241756\\
87.33261623828433    8.27484285\\
100.0    8.49258676\\
};
\addlegendentry{10-step O-GD}

\addplot [color=mycolor3, line width=1.5pt, dashdotted]
table[row sep=crcr]{%
1.0    0.00030381\\
1.1450475699382818    0.00000106191286\\
1.3111339374215643    0.00000113072221\\
1.5013107289081733    0.12129016\\
1.7190722018585745    0.37656921\\
1.9684194472866121    0.87785312\\
2.2539339047347906    1.55377757\\
2.580861540418075    2.16994952\\
2.955209235202887    2.52726977\\
3.383855153428234    2.79898134\\
3.874675120456131    3.09389539\\
4.436687330978613    3.46913934\\
5.0802180469130205    3.93903266\\
5.817091329374358    4.50533578\\
6.6608462908091575    5.17037083\\
7.626985859023444    5.93790975\\
8.733261623828433    6.80630687\\
10.0    7.76391582\\
11.450475699382817    8.77834008\\
13.111339374215643    9.78864754\\
15.013107289081734    10.71569587\\
17.190722018585745    11.51171614\\
19.68419447286612    12.26576476\\
22.539339047347912    13.2178792\\
25.808615404180742    14.57981388\\
29.552092352028865    16.39460693\\
33.83855153428233    18.64356542\\
38.7467512045613    21.32332037\\
44.366873309786115    24.48104467\\
50.8021804691302    28.154262\\
58.170913293743574    32.41401643\\
66.60846290809158    37.34888935\\
76.26985859023443    43.03342678\\
87.33261623828433    49.6187144\\
100.0    57.22239025\\
};
\addlegendentry{O-NM}

\addplot [color=mycolor2, line width=1.5pt, densely dotted]
  table[row sep=crcr]{%
1.0    0.00255234\\
1.1450475699382818    0.00000159781028\\
1.3111339374215643    0.00000221445264\\
1.5013107289081733    0.00195083\\
1.7190722018585745    0.0365759\\
1.9684194472866121    0.20454317\\
2.2539339047347906    0.47779318\\
2.580861540418075    0.89480687\\
2.955209235202887    1.47852995\\
3.383855153428234    2.08941005\\
3.874675120456131    2.67575454\\
4.436687330978613    3.28469756\\
5.0802180469130205    3.9605062\\
5.817091329374358    4.66481819\\
6.6608462908091575    5.47142162\\
7.626985859023444    6.40403563\\
8.733261623828433    7.48324889\\
10.0    8.72380527\\
11.450475699382817    10.12476965\\
13.111339374215643    11.65596077\\
15.013107289081734    13.24260748\\
17.190722018585745    14.75830855\\
19.68419447286612    16.0592777\\
22.539339047347912    17.05220663\\
25.808615404180742    17.70944773\\
29.552092352028865    18.04856881\\
33.83855153428233    18.15725892\\
38.7467512045613    18.14072668\\
44.366873309786115    18.0721087\\
50.8021804691302    17.98409248\\
58.170913293743574    17.90203382\\
66.60846290809158    17.81818069\\
76.26985859023443    17.74242776\\
87.33261623828433    17.67392263\\
100.0    17.61719912\\
};
\addlegendentry{O-AGD}

\end{axis}
\end{tikzpicture}%
    \end{subfigure}
    \hfill
    \vspace{0.5em}
    \begin{subfigure}{0.9\columnwidth}
%
%
\definecolor{mycolor1}{RGB}{204, 51, 63} 
\definecolor{mycolor2}{RGB}{68, 12, 84}
\definecolor{mycolor3}{RGB}{53, 95, 141}
\definecolor{mycolor4}{RGB}{34, 168, 132}
\definecolor{mycolor5}{RGB}{122, 210, 81}

\begin{tikzpicture}

\begin{axis}[%
width=2.5in,
height=1in,
at={(0.758in,0.481in)},
scale only axis,
xmode=log,
xmin=1,
xmax=100,
xminorticks=true,
xlabel style={font=\color{white!15!black}},
xlabel={$\frac{L}{m}$},
ymin=0,
ymax=50,
ylabel={$\gamma_3$},
ylabel style={font=\color{white!15!black}, yshift=0.18in},
axis background/.style={fill=white},
xmajorgrids,
xminorgrids,
ymajorgrids,
legend style={at={(0.015,0.97)}, nodes={scale=0.7, transform shape}, anchor=north west, legend cell align=left, align=left, draw=white!15!black}
]
\addplot[color=mycolor1, line width=1.5pt, solid] table[row sep=crcr] {
1.0    0.0\\
1.1450475699382818    5.783802639300411e-07\\
1.3111339374215643    7.569532650548262e-07\\
1.5013107289081733    9.050270964185993e-07\\
1.7190722018585745    9.84538124511524e-07\\
1.9684194472866121    1.0929683442010017e-06\\
2.2539339047347906    1.1275911497686031e-05\\
2.580861540418075    0.7412992460617795\\
2.955209235202887    2.1656066720392073\\
3.383855153428234    3.6820626152669513\\
3.874675120456131    4.846719732433918\\
4.436687330978613    5.832726549206101\\
5.0802180469130205    6.7978048414960535\\
5.817091329374358    7.789550082604983\\
6.6608462908091575    8.870113956272025\\
7.626985859023444    10.070314699972842\\
8.733261623828433    11.41618158212185\\
10.0    12.93407949307405\\
11.450475699382817    14.652559516651529\\
13.111339374215643    16.6029263997537\\
15.013107289081734    18.82079404232362\\
17.190722018585745    21.34592658302572\\
19.68419447286612    24.223390783338523\\
22.539339047347912    27.503503447486395\\
25.808615404180742    31.242971983364015\\
29.552092352028865    35.5047002382137\\
33.83855153428233    40.35811010446448\\
38.7467512045613    45.878814404969404\\
44.366873309786115    52.147608809711684\\
50.8021804691302    59.25137427913043\\
58.170913293743574    67.2827118989872\\
66.60846290809158    76.3462483457868\\
76.26985859023443    86.57046760457888\\
87.33261623828433    98.12164768070178\\
100.0    111.20251832486933\\
};
\addlegendentry{O-GD}

\addplot[color=mycolor4,line width=1.5pt, dashed] table[row sep=crcr] {
1.0    0.0\\
1.1450475699382818    1.6982911278029824e-06\\
1.3111339374215643    2.5955006659415633e-06\\
1.5013107289081733    3.419398192008019e-06\\
1.7190722018585745    5.550764809559243e-06\\
1.9684194472866121    8.264151064496548e-06\\
2.2539339047347906    1.1387522466955172e-05\\
2.580861540418075    1.7062634865750947e-05\\
2.955209235202887    3.031890579390313e-05\\
3.383855153428234    0.001656675177260568\\
3.874675120456131    0.48458504691586657\\
4.436687330978613    0.6217146895820097\\
5.0802180469130205    0.7217461022896613\\
5.817091329374358    0.8000571034601531\\
6.6608462908091575    0.8564936981332185\\
7.626985859023444    0.8893517166147984\\
8.733261623828433    0.8974417212811283\\
10.0    0.8811533726513995\\
11.450475699382817    0.8443220545991672\\
13.111339374215643    0.7940762103104306\\
15.013107289081734    0.7368829712334486\\
17.190722018585745    0.6771696270406357\\
19.68419447286612    0.6175921221562334\\
22.539339047347912    0.5596722794339493\\
25.808615404180742    0.5045756356808166\\
29.552092352028865    0.4528207683341554\\
33.83855153428233    0.4049074254220005\\
38.7467512045613    0.36088706853871827\\
44.366873309786115    0.32084011722749856\\
50.8021804691302    0.28434686701668876\\
58.170913293743574    0.2521414867578857\\
66.60846290809158    0.22443996044168638\\
76.26985859023443    0.19696247885549703\\
87.33261623828433    0.17538582189881866\\
100.0    0.0\\
};
\addlegendentry{2-step O-GD}

\addplot[color=mycolor5, line width=1.5pt, dotted] table[row sep=crcr] {
1.0    0.0\\
1.1450475699382818    6.344376188355779e-06\\
1.3111339374215643    2.729621725524669e-05\\
1.5013107289081733    2.1404764893144025e-05\\
1.7190722018585745    3.8064510187709184e-05\\
1.9684194472866121    5.958641440272167e-06\\
2.2539339047347906    1.1452845398020685e-05\\
2.580861540418075    1.3670956972734271e-05\\
2.955209235202887    1.4264054643557795e-05\\
3.383855153428234    1.960171976991646e-05\\
3.874675120456131    2.417299959450026e-05\\
4.436687330978613    2.9218338504941465e-05\\
5.0802180469130205    3.7190557401073826e-05\\
5.817091329374358    5.241647688876258e-05\\
6.6608462908091575    8.72716835580186e-05\\
7.626985859023444    0.03455752689649248\\
8.733261623828433    0.35942447082384743\\
10.0    0.5459902162755664\\
11.450475699382817    0.6724305000891774\\
13.111339374215643    0.7809127896782979\\
15.013107289081734    0.879439348875649\\
17.190722018585745    0.9660508789001276\\
19.68419447286612    1.0395680747614973\\
22.539339047347912    1.097141961045817\\
25.808615404180742    1.1363089021730348\\
29.552092352028865    1.1506824014982178\\
33.83855153428233    1.1402596120675594\\
38.7467512045613    1.1067381665205913\\
44.366873309786115    1.054515944723613\\
50.8021804691302    0.9895962087795935\\
58.170913293743574    0.9183253840181769\\
66.60846290809158    0.8454413859254654\\
76.26985859023443    0.7719262766113485\\
87.33261623828433    0.700930122765774\\
100.0    0.6334562491858783\\
};
\addlegendentry{10-step O-GD}

\addplot[color=mycolor3, line width=1.5pt, dash dot] table[row sep=crcr] {
1.0    0.0\\
1.1450475699382818    4.3315716672652116e-07\\
1.3111339374215643    8.990628131174396e-07\\
1.5013107289081733    0.10283819743611668\\
1.7190722018585745    0.30497299005241024\\
1.9684194472866121    0.6589212759832014\\
2.2539339047347906    1.0923067858952689\\
2.580861540418075    1.52920607092328\\
2.955209235202887    1.9023570924551307\\
3.383855153428234    2.291426789899823\\
3.874675120456131    2.7323276417215268\\
4.436687330978613    3.235378418725581\\
5.0802180469130205    3.7992209984566543\\
5.817091329374358    4.421581082618054\\
6.6608462908091575    5.10118891335007\\
7.626985859023444    5.834792949097069\\
8.733261623828433    6.612391649349733\\
10.0    7.411964873073623\\
11.450475699382817    8.192460107567939\\
13.111339374215643    8.890189759868578\\
15.013107289081734    9.430232065330653\\
17.190722018585745    9.772972179783354\\
19.68419447286612    9.99129888498347\\
22.539339047347912    10.269844272752911\\
25.808615404180742    10.757160901070963\\
29.552092352028865    11.459876709767872\\
33.83855153428233    12.331908190556144\\
38.7467512045613    13.338599451249072\\
44.366873309786115    14.46428550068926\\
50.8021804691302    15.699347007367285\\
58.170913293743574    17.043737217156057\\
66.60846290809158    18.49693849839729\\
76.26985859023443    20.062059775083014\\
87.33261623828433    21.74312799776845\\
100.0    23.547124786332645\\
};
\addlegendentry{O-NM}

\addplot[color=mycolor2,line width=1.5pt, densely dotted] table[row sep=crcr] {
1.0    0.0\\
1.1450475699382818    2.5159269565722228e-08\\
1.3111339374215643    5.051967978528258e-07\\
1.5013107289081733    0.29554667298937054\\
1.7190722018585745    0.5357880393471705\\
1.9684194472866121    0.9478095915077166\\
2.2539339047347906    1.559261009188491\\
2.580861540418075    2.407050132261089\\
2.955209235202887    3.5349126961258395\\
3.383855153428234    4.708215380422243\\
3.874675120456131    5.821884482521034\\
4.436687330978613    6.96367676043836\\
5.0802180469130205    8.20062526567647\\
5.817091329374358    9.55431777473381\\
6.6608462908091575    11.086900557721105\\
7.626985859023444    12.818798234280854\\
8.733261623828433    14.768198846639251\\
10.0    16.953628441438003\\
11.450475699382817    19.394507890833214\\
13.111339374215643    22.103590773832746\\
15.013107289081734    25.059399544163597\\
17.190722018585745    28.155083645568116\\
19.68419447286612    31.15872589063841\\
22.539339047347912    33.74395871384359\\
25.808615404180742    35.60828861856121\\
29.552092352028865    36.667126782162235\\
33.83855153428233    37.12084372411895\\
38.7467512045613    37.24088606978028\\
44.366873309786115    37.21056070892256\\
50.8021804691302    37.12024122214787\\
58.170913293743574    37.017273004181014\\
66.60846290809158    36.90779507021773\\
76.26985859023443    36.804737621021715\\
87.33261623828433    36.71555773798801\\
100.0    36.61810823433964\\
};
\addlegendentry{O-AGD}

\end{axis}
\end{tikzpicture}%
    \end{subfigure}
    \caption{\rev{Regret sensitivity to $\mathcal{S}_T$ $(\gamma_1)$, $\mathcal{G}_T$ $(\gamma_2)$ and $\mathcal{V}_T$ $(\gamma_3)$ over function class condition ratio, according to Theorem \ref{thm:dynamic}.}}
    \label{fig:results_variational}
\end{figure}
\ifthenelse{\boolean{arxiv}}{\section{Extension to parameter-varying algorithms}
\label{sec:LPV}

So far we have only considered static algorithms, in the sense that the parametrization $A, B, C, D$ were time-invariant matrices. In line with \cite{jakob_iannelli}, we can also cover time-varying algorithms. To motivate this, note that the strong convexity and smoothness constants are typically used for the tuning of algorithm parameters. In practice, when those constants are time-varying, one can also resort to algorithms with time-varying parameters, such as stepsizes for example. 
A way to capture such generalized algorithms is to consider (\ref{eq:generalized_algorithm:state_space}) as a linear parameter-varying (LPV) system
\begin{align}
    \label{eq:lpv_system}
    \begin{aligned}
    \xi_{t+1} &= A(\theta_t) \xi_t + B(\theta_t) u_t \\
    y_t &= C(\theta_t) \xi_t + D(\theta_t) u_t.
    \end{aligned}
\end{align}
Here, $\theta_t$ are parameters from some compact parameter domain $\Theta$. Formulation (\ref{eq:lpv_system}) allows for instance to regard $m_t$ or $L_t$ as explicit parameters, or nonlinear adaption laws as e.g. in \cite{hu_neurips}. We can also account for the cases where lower and upper bounds on the parameter variations are available, i.e. $v_{\min} \leq \Delta \theta_t \leq v_{\max}$ for $\Delta \theta_t := \theta_t - \theta_{t+1}$.

It can be shown that the results presented in the previous section hold also for the LPV setup with only minor modifications. We state the following extension of Theorem~\ref{thm:dynamic}.

\begin{proposition}
    Consider problem (\ref{eq:composite_optimization}) with algorithm (\ref{eq:generalized_algorithm}) but LPV realization (\ref{eq:lpv_system}). Let Assumptions \ref{assum:fixed_point} and \ref{assum:Bu_Du} hold uniformly for all $\theta \in \Theta$. Moreover, let the IQC realization of $\Psi_f$ and the augmented plant (\ref{eq:augmented_plant}) have parameter dependent realizations, indicated by a supscript $\theta$. If there exists a symmetric matrix valued function $P_\theta~:=~P(\theta):~\Theta~\rightarrow~\mathbb{S}^{n_\xi+n_\zeta}$, non-negative vectors $\lambda_p \in \mathbb{R}_{\geq 0}^p, \lambda_q \in \mathbb{R}_{\geq 0}^q$ and scalars \rev{$\gamma_{\Delta \xi}, \gamma_{\Delta \delta} > 0$}, such that $P(\theta) \succ 0$ and the LMI
    {\setlength{\arraycolsep}{1pt}
    \rev{
    \begin{multline}
        \label{eq:LMI:variational}
        \hspace{-12pt}
        \begin{bmatrix}
            \star
        \end{bmatrix}^\top
        \hspace{-5pt}
        \begin{bmatrix}
            -P_{\theta^+} &  &  & & \\
            & P_{\theta} & & & \\
            & & M_\lambda & & \\
            & & & -\gamma_{\Delta \xi} I & \\
            & & & &  -\gamma_{\Delta \delta} I
        \end{bmatrix}
        \hspace{-5pt}
        \begin{bmatrix}
        I & 0 & 0 & 0\\
        \hat{A}_\theta & \hat{B}_{\theta,u} & \hat{B}_{\theta, \Delta \xi} & \hat{B}_{\theta,\Delta \delta} \\ \hat{C}_\theta & \hat{D}_{\theta,u} & \hat{D}_{\theta,\Delta \xi} & \hat{D}_{\theta,\Delta \delta} \\
        0 & 0 & I & 0 \\ 0 & 0 & 0 & I
        \end{bmatrix}
        \\ + \frac{1}{2} 
        \begin{bmatrix}
            0 & \bigl[ \begin{smallmatrix} C_1^\top & 0 \\ 0 & 0 \end{smallmatrix} \bigr] \\ \bigl[ \begin{smallmatrix} C_1 & 0 \\ 0 & 0 \end{smallmatrix} \bigr] & 0
        \end{bmatrix}
        \preceq 0 
    \end{multline}
    }
    }
    holds with $M_\lambda = \mathrm{blkdiag}(M_2 \otimes \mathrm{diag}(\lambda_p), M_1 \otimes \mathrm{diag}(\lambda_q))$, $P_{\theta^+} := P(\theta + \Delta \theta)$ for all $\theta \in \Theta$ and all possible variations $v_{\min} \leq \Delta \theta \leq v_{\max}$ then we have the same regret bound as in (\ref{eq:regret:dynamic:bigO}) with $m=\inf_t m_t$ and $L=\sup_t L_t$.
\end{proposition}

The proof is in line with Theorem \ref{thm:dynamic} and the methodologies in \cite{jakob_iannelli}.
Note however, that the readout matrix $C_1$ has to stay parameter-invariant, since $x_t$ is only allowed to depend on information up to $t-1$. 
We leave the exploration of this framework, for instance to analyze adaptive OCO algorithms, for future works.

}{}
\section{Conclusion}
\label{sec:conclusions}

In this paper, we have presented a novel framework \rev{for an automated regret analysis of first-order algorithms in OCO}. By recasting the algorithm as a feedback interconnection of a linear system and a time-varying \rev{oracle} we are able to provide an alternative proof strategy and a computational tool to quantify the regret of generalized first-order optimization algorithms. This new analysis framework represents a new viewpoint on OCO and can contribute to obtain a more systematic way to show regret. Future work may include the use of further robust control tools, such as 
algorithm synthesis or robustness analyses for gradient errors.
\ifthenelse{\boolean{arxiv}}{
    \appendix
    \section{}

\subsection{On the vIQC for the normal cone}\label{appendix:A}

Note that even though $\mathcal{I}_\mathcal{X}$ (and thus $\mathcal{N}_\mathcal{X}$) is not time-varying, we still \emph{cannot} recover conventional dynamic IQCs that have been developed for static passive operators. Since dynamic IQCs incorporate memory, the variation of the minimizer $x_t^\star$ has still to be accounted for. 

Since $\mathcal{I}_\mathcal{X}$ is proper, closed and convex, we can apply Proposition \ref{prop:vIQC}, letting $f_t \leftarrow \mathcal{I}_{\mathcal{X}}$ and $\nabla f_t \leftarrow \mathcal{N}_\mathcal{X}$, in the limit $m \rightarrow 0$ and $L \rightarrow \infty$. I.e., we expand the left hand side of (\ref{eq:IQC:delta:quadratic_inequality}), multiply the inequality by $\frac{1}{L}$, set $m=0$ and let $L\rightarrow \infty$. Moreover, since $\mathcal{I}_\mathcal{X}$ is static, the right hand side boils down to zero, and we get
\begin{equation*}
    \sum_{t=1}^T \hat{\psi}_t^\top M_1 \hat{\psi}_t \geq 0
\end{equation*}
for the reduced filter realization
\begin{equation}
    \hat{\psi}_t = \underbrace{\left[ \begin{array}{c|ccc} 0 & I_d & 0 & I_d \\ \hline  -I_d & I_d & 0 & 0 \\ 0 & 0 & I_d & 0 \end{array} \right]}_{=:\Psi_\mathcal{I}} \begin{bmatrix} x_t - x_t^\star \\ \beta_t \\ \Delta x_t^\star \end{bmatrix},
\end{equation}
where $\beta_t \in \mathcal{N}_\mathcal{X}(x_t)$. This gives a possible realization of $\Psi_\mathcal{I}$.

\subsection{On building the augmented plant}\label{appendix:B}

We make some comments for the sake of implementation. Note that in practice, one might combine different types of IQCs for the components of $(u_t, y_t)$. I.e., we may use both a pointwise and vIQC for the very same component, as well as IQCs for repeated nonlinearities \cite{amato}. Ultimately, we get a set of filter outputs $(\psi_t^1; \dots; \psi_t^k, \dots, \psi_t^K)$ (minimum $p+q$), comprising all pointwise/variational/repeated IQCs, which are each generated by one filter $\Psi_k$, $k\in\mathbb{I}_K$. The multiplier matrix $M_\lambda$ has respective multiplier $\lambda_k M_k$, $M_k \in \{ M_1, M_2 \}$ on its diagonal, and may also have cross terms depending on repeated IQCs.
The vertical stack of all filters $\Psi_k$ gives the concatenated filter realization
\begin{equation*}
\left[ \begin{array}{c|c} \hat{A}_\Psi & \hat{B}_\Psi \\ \hline \hat{C}_\Psi & \hat{D}_\psi \end{array} \right],
\end{equation*}
mapping all inputs $\delta_t^i, g_t^j, \Delta x_t^\star, \mathbf{\Delta \delta}(s_t^i)$, $i\in\mathbb{I}_p$, $j\in\mathbb{I}_q$, to all IQC components $\psi_t^k$, $k\in\mathbb{I}_K$. An input entry selector matrix $S_\mathrm{in}$ ensures that the output order matches the multiplier order in $M_\lambda$, and an output permutation matrix $S_\mathrm{out}$ is used to correctly pick the individual inputs from the stacked input vector $\mathrm{vec}(\tilde{y}_t, u_t, \Delta \xi_t^\star, \bm{\Delta \delta}_t)$ and to realize $\Delta x_t^\star$ with the algorithm output matrix $C$. We arrive at the final compact IQC
\begin{equation*}
\left[ \begin{array}{c|c} {A}_\Psi & {B}_\Psi \\ \hline {C}_\Psi & {D}_\psi \end{array} \right] = \left[ \begin{array}{c|c} \hat{A}_\Psi & \hat{B}_\Psi S_\mathrm{in} \\ \hline S_\mathrm{out} \hat{C}_\Psi & S_\mathrm{out} \hat{D}_\psi S_\mathrm{in} \end{array} \right],
\end{equation*}
which is then used to realize (\ref{eq:augmented_plant})

}{}

\bibliographystyle{ieeetr}
\bibliography{root}

\end{document}